\documentclass{amsart}
\usepackage{amssymb}
\usepackage{amsmath}
\usepackage{amsfonts}

\newtheorem{theorem}{Theorem}
\theoremstyle{plain}

\newtheorem{example}{Example}

\newtheorem{lemma}{Lemma}

\newtheorem{proposition}{Proposition}

\numberwithin{equation}{section}
\input{tcilatex}

\begin{document}
\title{ YAMABE SOLITONS ON THREE-DIMENSIONAL NORMAL ALMOST PARACONTACT
METRIC MANIFOLDS }
\author{I. K\"{u}peli Erken}
\address{Faculty of Natural Sciences, Architecture and Engineering,
Department of Mathematics, Bursa Technical University, Bursa, TURKEY}
\email{irem.erken@btu.edu.tr}

\begin{abstract}
The purpose of the paper is to study Yamabe solitons on three-dimensional
para-Sasakian, paracosymplectic and para-Kenmotsu manifolds. Mainly, we
proved that

$\bullet $If the semi-Riemannian metric of a three-dimensional para-Sasakian
manifold is a Yamabe soliton, then it is of constant scalar curvature, and
the flow vector field $V$ is Killing. In the next step, we proved that
either manifold has constant curvature $-1$ and reduces to an Einstein
manifold, or $V$ is an infinitesimal automorphism of the paracontact metric
structure on the manifold.

$\bullet $If the semi-Riemannian metric of a three-dimensional
paracosymplectic manifold is a Yamabe soliton, then it has constant scalar
curvature. Furthermore either manifold is $\eta $-Einstein, or Ricci flat.

$\bullet $ If the semi-Riemannian metric on a three-dimensional
para-Kenmotsu manifold is a Yamabe soliton, then the manifold is of constant
sectional curvature $-1$, reduces to an Einstein manifold. Furthermore,
Yamabe soliton is expanding with $\lambda =-6$ and the vector field $V$ is
Killing.

Finally, we construct examples to illustrate the results obtained in
previous sections.
\end{abstract}

\maketitle

\section{\textbf{Introduction}}

\label{introduction}

Several years ago, the notion of Yamabe flow was introduced by Richard
Hamilton at the same time as the Ricci flow, (see \cite{HAM1}, \cite{HAM2})
as a tool for constructing metrics of constant scalar curvature in a given
conformal class of Riemannian metrics on $(M^{n},g)(n\geq 3)$. On a smooth
semi-Riemannian manifold, Yamabe flow can be defined as the evolution of the
semi-Riemannian metric $g_{0}$ in time $t$ to $g=g(t)$ by means of the
equation%
\begin{equation*}
\frac{\partial }{\partial t}g=-rg,\text{ }g(0)=g_{0},
\end{equation*}%
where $r$ denotes the scalar curvature corresponds to $g$.

The significance of Yamabe flow lies in the fact that it is a natural
geometric deformation to metrics of constant scalar curvature. One notes
that Yamabe flow corresponds to the fast diffusion case of the porous medium
equation (the plasma equation) in mathematical physics. In dimension $n=2$
the Yamabe flow is equivalent to the Ricci flow (defined by $\frac{\partial 
}{\partial t}g(t)=-2\rho (t)$, where $\rho $ stands for the Ricci tensor).
However in dimension $n>2$ the Yamabe and Ricci flows do not agree, since
the first one preserves the conformal class of the metric but the Ricci flow
does not in general. Just as a Ricci soliton is a special solution of the
Ricci flow, a Yamabe soliton is a special solution of the Yamabe flow that
moves by one parameter family of diffeomorphisms $\varphi _{t}$ generated by
a fixed (time-independent) vector field $V$ on $M$, and homotheties, i.e. $%
g(.,t)=\sigma (t)\varphi _{t}$ $^{\ast }g_{0}$.

A semi-Riemannian manifold $(M^{n},g)$ is a Yamabe soliton \cite{CHOW} if it
admits a vector field $V$ such that%
\begin{equation}
\mathcal{L}_{V}g=(\lambda -r)g,  \label{yamabe soliton}
\end{equation}%
where $\mathcal{L}_{V}$ denotes the Lie derivative in the direction of the
vector field $V$ and $\lambda $ is a real number. Moreover, a vector field $%
V $ as in the definition is called a \textit{soliton vector field} (for $%
(M^{n},g)$). In the particular case of $V$ being a gradient, i.e., $V$ $%
=\nabla f$ for some potential function $f$, $(M^{n},g)$ is said to be a 
\textit{gradient Yamabe soliton} and $V$ is called a \textit{gradient
soliton vector field }(for $(M^{n},g)$). A Yamabe soliton is said to be 
\textit{shrinking, steady} or \textit{expanding} if it admits a soliton
vector field for which, respectively, $\lambda >0,$ $\lambda =0,~\lambda <0$.

A\textit{\ Ricci soliton} (see \cite{CHOW}) is a natural generalization of
Einstein metric (that is, the Ricci tensor is a constant multiple of the
semi-Riemannian metric $g)$. A Ricci soliton $(g,V,\lambda )$ is defined on
a semi-Riemannian manifold $(M^{n},g)$ by%
\begin{equation}
(\mathcal{L}_{V}g)(X,Y)+2S(X,Y)+2\mu g(X,Y)=0,  \label{Ricci Soliton}
\end{equation}%
where $\mathcal{L}_{V}g$ denotes the Lie derivative of semi-Riemannian
metric $g$ along a vector field $V$, $\mu $ is a constant, and $X,Y$ are
arbitrary vector fields on $M$. It is clear that a Ricci soliton with $V$
zero or a Killing vector field reduces to an Einstein metric. A Ricci
soliton is said to be \textit{shrinking, steady }and \textit{expanding}
according as $\mu $ is\textit{\ negative, zero }and\textit{\ positive},
respectively. Also, Ricci solitons have been studied extensively in the
context of semi-Riemannian geometry; we may refer to \cite{bejan}, \cite%
{CALVARUSO 1}, \cite{CALVARUSO 2}, \cite{CALVARUSO 3}, \cite{Pina} and
references therein.

Yamabe solitons coincide with Ricci solitons (defined by (\ref{Ricci Soliton}%
)) in dimension $n=2$. In higher dimensions Ricci solitons and Yamabe
solitons have different behaviours. For instance, since any soliton vector
field is a conformal vector field, if the scalar curvature is constant then
it must be necessarily zero unless the soliton vector field is Killing [\cite%
{kuhnel}, Corollary 2.2(i)].

In this study, we make the first contribution to investigate Yamabe solitons
on paracontact geometry. Yamabe solitons on three dimensional Sasakian
manifolds and Kenmotsu manifolds were studied respectively by R. Sharma \cite%
{sharma} and Y. Wang \cite{wang}.

The outline of the article goes as follows. In Section $2$, we recall basic
facts and some results related with almost paracontact manifolds which we
will need throughout the paper. Section $3$ is devoted to Yamabe solitons on
three dimensional para-Sasakian manifolds. Our first main result about
para-Sasakian manifolds is that if the semi-Riemannian metric of a
three-dimensional para-Sasakian manifold is a Yamabe soliton, then it has
constant scalar curvature, and the vector field $V$ is Killing. Furthermore,
either manifold has constant curvature $-1$ and reduces to an Einstein
manifold, or $V$ is an infinitesimal automorphism of the paracontact metric
structure on the manifold$.$ Our second main result about para-Sasakian
manifolds is that if manifold admits a Yamabe soliton and $V$ is pointwise
collinear vector field with the structure vector field $\xi $, then $V$ is a
constant multiple of $\xi $. Also, we show that if the semi-Riemannian
metric of a three dimensional para-Sasakian manifold with $r=-6$ is a Yamabe
soliton, then it is also a Ricci soliton. Section $4$ is devoted to Yamabe
solitons on three dimensional paracosymplectic manifolds. Our main result
about paracosymplectic manifolds is that if the semi-Riemannian metric of a
three-dimensional paracosymplectic manifold is a Yamabe soliton, then it has
constant scalar curvature. Furthermore, either manifold is $\eta $-Einstein,
or Ricci flat. Section $5$ is devoted to Yamabe solitons on three
dimensional para-Kenmotsu manifolds. Our first main result about
para-Kenmotsu manifolds is that if the semi-Riemannian metric on a
three-dimensional para-Kenmotsu manifold is a Yamabe soliton, then the
manifold is of constant sectional curvature $-1$, reduces to an Einstein
manifold. Furthermore, Yamabe soliton is expanding with $\lambda =-6$ and
the vector field $V$ is Killing. Our second main result about para-Kenmotsu
manifolds is that if the semi-Riemannian metric of a three dimensional
para-Kenmotsu manifold is a Yamabe soliton, then it is also a Ricci soliton.
At last, we show that if three-dimensional para-Kenmotsu manifold admits a
Yamabe soliton for a vector field $V$ and a constant $\lambda $, then $V$
can not be pointwise collinear with $\xi $. Finally, we construct examples
to illustrate the results obtained in previous sections.

\section{Preliminaries}

\label{preliminaries}

In this section we collect the formulas and results we need on paracontact
metric manifolds. All manifolds are assumed to be connected and smooth. We
may refer to \cite{kaneyuki1}, \cite{Za} and references therein for more
information about paracontact metric geometry.

Paracontact metric structures were introduced in \cite{kaneyuki1}, as a
natural odd-dimensional counterpart to paraHermitian structures, like
contact metric structures correspond to the Hermitian ones. Paracontact
metric manifolds $(M^{2n+1},\varphi ,\xi ,\eta ,g)$ have been studied by
many authors in the recent years, particularly since the appearance of \cite%
{Za}. The curvature identities for different classes of almost paracontact
metric manifolds were obtained e.g. in \cite{DACKO}, \cite{Welyczko}, \cite%
{Za}. The importance of paracontact geometry, and in particular of
para-Sasakian geometry, has been pointed out especially in the last years by
several papers highlighting the interplays with the theory of para-K\"{a}%
hler manifolds and its role in semi-Riemannian geometry and mathematical
physics (cf. e.g.,\cite{Biz} ,\cite{kupmur}).

An $(2n+1)$-dimensional smooth manifold $M$ is said to have an \emph{almost
paracontact structure} if it admits a $(1,1)$-tensor field $\varphi $, a
vector field $\xi $ and a $1$-form $\eta $ satisfying the following
conditions:

\begin{enumerate}
\item[(i)] $\eta (\xi )=1$,\ $\varphi ^{2}=I-\eta \otimes \xi $,

\item[(ii)] the tensor field $\varphi $ induces an almost paracomplex
structure on each fibre of ${\mathcal{D}}=\ker (\eta )$, i.e. the $\pm 1$%
-eigendistributions, ${\mathcal{D}}^{\pm }={\mathcal{D}}_{\varphi }(\pm 1)$
of $\varphi $ have equal dimension $n$.
\end{enumerate}

From the definition it follows that $\varphi \xi =0$, $\eta \circ \varphi =0$
and the endomorphism $\varphi $ has rank $2n$. We denote by $[\varphi
,\varphi ]$ the Nijenhuis torsion 
\begin{equation*}
\lbrack \varphi ,\varphi ](X,Y)=\varphi ^{2}[X,Y]+[\varphi X,\varphi
Y]-\varphi \lbrack \varphi X,Y]-\varphi \lbrack X,\varphi Y].
\end{equation*}%
When the tensor field $N_{\varphi }=[\varphi ,\varphi ]-2d\eta \otimes \xi $
vanishes identically the almost paracontact manifold is said to be \emph{%
normal}. If an almost paracontact manifold admits a semi-Riemannian metric $%
g $ such that 
\begin{equation}
g(\varphi X,\varphi Y)=-g(X,Y)+\eta (X)\eta (Y),  \label{G METRIC}
\end{equation}%
for all $X,Y\in \chi (M)$, then we say that $(M,\varphi ,\xi ,\eta ,g)$ is
an \emph{almost paracontact metric manifold}. Notice that any such a
semi-Riemannian metric is necessarily of signature $(n+1,n)$. For an almost
paracontact metric manifold, there always exists an orthogonal basis $%
\{X_{1},\ldots ,X_{n},Y_{1},\ldots ,Y_{n},\xi \}$, such that $%
g(X_{i},X_{j})=\delta _{ij}$, $g(Y_{i},Y_{j})=-\delta _{ij}$, $%
g(X_{i},Y_{j})=0$, $g(\xi ,X_{i})=g(\xi ,Y_{j})=0$, and $Y_{i}=\varphi X_{i}$%
, for any $i,j\in \left\{ 1,\ldots ,n\right\} $. Such basis is called a $%
\varphi $-basis.

We can now define the \textit{fundamental 2-form} of the almost paracontact
metric manifold by $\Phi (X,Y)=g(X,\varphi Y)$. If $d\eta (X,Y)=$ $%
g(X,\varphi Y)$, then $(M,\varphi ,\xi ,\eta ,g)$ is said to be \emph{%
paracontact metric manifold }. In a paracontact metric manifold one defines
a symmetric, trace-free operator $h=\frac{1}{2}{\mathcal{L}}_{\xi }\varphi $%
, where $\mathcal{L}_{\xi }$, denotes the Lie derivative. It is known \cite%
{Za} that $h$ anti-commutes with $\varphi $ and satisfies $h\xi =0,$ tr$h=$tr%
$h\varphi =0$ and 
\begin{equation}
\nabla \xi =-\varphi +\varphi h,  \label{nablaxi}
\end{equation}%
where $\nabla $ is the Levi-Civita connection of the semi-Riemannian
manifold $(M,g)$.

Moreover $h=0$ if and only if $\xi $ is Killing vector field. In this case $%
(M,\varphi ,\xi ,\eta ,g)$ is said to be a \emph{K-paracontact manifold}. A
normal paracontact metric manifold is called a \emph{para-Sasakian manifold}%
. Also in this context the para-Sasakian condition implies the $K$%
-paracontact condition and the converse holds only in dimension $3$. We also
recall that any para-Sasakian manifold satisfies 
\begin{eqnarray}
R(X,Y)\xi &=&-(\eta (Y)X-\eta (X)Y),  \label{Pasa} \\
(\nabla _{X}\varphi )Y &=&-g(X,Y)\xi +\eta (Y)X,  \label{nablafi} \\
\nabla _{X}\xi &=&-\varphi X,  \label{nablaksi} \\
R(\xi ,X)Y &=&-g(X,Y)\xi +\eta (Y)X,  \label{R} \\
S(X,\xi ) &=&-(n-1)\eta (X),  \label{S}
\end{eqnarray}%
where $Q$ is the Ricci operator, $R$ is the Riemannian curvature tensor and $%
S$ is Ricci tensor defined by $S(X,Y)=g(QX,Y).$

From \cite{oneill} we know that if $M$ has constant curvature $c,$ then%
\begin{equation}
R(X,Y)Z=c(g(Z,X)Y-g(Z,Y)X).  \label{constant curvature}
\end{equation}%
On an almost paracontact metric manifold $M$, if the Ricci operator satisfies%
\begin{equation}
Q=\alpha id+\beta \eta \otimes \xi ,  \label{einstein}
\end{equation}%
where both $\alpha $ and $\beta $ are smooth functions, then the manifold is
said to be an $\eta $\textit{-Einstein manifold. }An $\eta $-Einstein
manifold with $\beta $ vanishing and $\alpha $ a constant is obviously an
Einstein manifold. An $\eta $-Einstein manifold is said to be \textit{proper 
}$\eta $\textit{-Einstein} if $\beta \neq 0$.

We recall that the curvature tensor of a $3$-dimensional semi-Riemannian
manifold satisfies 
\begin{equation}
R(X,Y)Z=g(Y,Z)QX-g(X,Z)QY+g(QY,Z)X-g(QX,Z)Y-\frac{r}{2}(g(Y,Z)X-g(X,Z)Y)
\label{THREE DIM CURVATURE}
\end{equation}%
where $Q$ is the Ricci operator of $M$.

An infinitesimal automorphism is a vector field such that Lie derivatives
along it of all objects of some tensor structure vanish. For an almost
paracontact, metric structure, the condition that a vector $V$ is an
infinitesimal automorphism is as follows:%
\begin{equation}
{\mathcal{L}}_{V}\eta ={\mathcal{L}}_{V}\xi ={\mathcal{L}}_{V}\Phi ={%
\mathcal{L}}_{V}g=0.  \label{infi}
\end{equation}

\section{Yamabe Solitons on three dimensional Para-Sasakian manifolds}

In this section, before presenting our main results about Yamabe Solitons on
three dimensional para-Sasakian manifolds, we will give some lemmas which
will be used later.

A vector field $V$ on an $n$-dimensional semi-Riemannian manifold $(M,g)$ is
said to be \textit{conformal vector field }if,%
\begin{equation}
{\mathcal{L}}_{V}g=2\rho g,  \label{conformal}
\end{equation}%
where $\rho $ is called the \textit{conformal coefficient }(from (\ref%
{yamabe soliton}) we get $\rho =\frac{\lambda -r}{2})$. If conformal
coefficient is zero, we will say that conformal vector field is \textit{%
Killing vector field.}

\begin{lemma}
\cite{yano}\label{1}On an $n$-dimensional semi-Riemann manifold $(M^{n},g)$
endowed with a conformal vector field $V$, we have%
\begin{eqnarray}
({\mathcal{L}}_{V}S)(X,Y) &=&-(n-2)g(\nabla _{X}D_{\rho },Y)+(\Delta \rho
)g(X,Y),  \label{mert1} \\
{\mathcal{L}}_{V}r &=&-2\rho r+2(n-1)\Delta \rho  \label{mert2}
\end{eqnarray}%
for any vector fields $X$ and $Y,$ where $D$ denotes the gradient operator
and $\Delta :=-\func{div}D$ denotes the Laplacian operator of $g$.
\end{lemma}

\begin{lemma}
\label{2}For a para-Sasakian manifold, the following relations are valid:%
\begin{eqnarray}
(i)\eta ({\mathcal{L}}_{V}\xi ) &=&\frac{r-\lambda }{2},  \label{mert3} \\
(ii)({\mathcal{L}}_{V}\eta )(\xi ) &=&\frac{\lambda -r}{2}.  \label{mert4}
\end{eqnarray}
\end{lemma}

\begin{proof}
Since the Reeb vector field $\xi $ is a unit vector field we have $g(\xi
,\xi )=1$. Taking the Lie-derivative of this relation along the vector field 
$V$ and using (\ref{yamabe soliton}) we get $(i)$. Using $\eta (\xi )=1$ and 
$(i)$, we have $(ii)$.
\end{proof}

\begin{lemma}
\label{2.5}For any three-dimensional para-Sasakian manifold $(M^{3},\varphi
,\xi ,\eta ,g)$, we have%
\begin{equation}
\xi (r)=0.  \label{mer4.1}
\end{equation}
\end{lemma}

\begin{proof}
If we replace $Y=Z$ by $\xi $ in (\ref{THREE DIM CURVATURE}) and use (\ref%
{Pasa}), (\ref{S}) we get%
\begin{equation}
QX=\left( \frac{r}{2}+1\right) X-\left( \frac{r}{2}+3\right) \eta (X)\xi
\label{MERT4.2}
\end{equation}%
for any vector field $X\in \chi (M)$. If we use (\ref{MERT4.2}), (\ref%
{nablaksi}) and (\ref{S}) in the following well known formula for
semi-Riemannian manifolds 
\begin{equation*}
trace\left\{ Y\rightarrow (\nabla _{Y}Q)X\right\} =\frac{1}{2}\nabla _{X}r
\end{equation*}%
we obtain%
\begin{equation}
\xi (r)=0.  \label{mert4.3}
\end{equation}
\end{proof}

From (\ref{MERT4.2}), we get

\begin{lemma}
\label{3}For a three-dimensional para-Sasakian manifold $(M^{3},\varphi ,\xi
,\eta ,g)$, the Ricci tensor $S$ is given by%
\begin{equation}
S(X,Y)=\left( \frac{r}{2}+1\right) g(X,Y)-\left( \frac{r}{2}+3\right) \eta
(X)\eta (Y)  \label{MERT5}
\end{equation}%
for any vector fields $X,Y\in \chi (M)$.
\end{lemma}

\begin{lemma}
\label{4}Suppose that the semi-Riemannian metric of a three-dimensional
para-Sasakian manifold $(M^{3},\varphi ,\xi ,\eta ,g)$ is a Yamabe soliton.
If the scalar curvature of $M^{3}$ is harmonic, that is $\Delta r=0$, then $%
\lambda =r.$
\end{lemma}

\begin{proof}
Since $V$ is a conformal vector field with $\rho =\frac{\lambda -r}{2}$, for 
$n=3$ from equations (\ref{mert1}) and (\ref{mert2}), we obtain%
\begin{eqnarray}
({\mathcal{L}}_{V}S)(X,Y) &=&\frac{1}{2}g(\nabla _{X}D_{r},Y)-\frac{1}{2}%
(\Delta r)g(X,Y),  \label{mert6} \\
{\mathcal{L}}_{V}r &=&r(r-\lambda )-2\Delta r  \label{mert7}
\end{eqnarray}%
for any vector fields $X,Y\in \chi (M)$. If we take the Lie-derivative of (%
\ref{MERT5}) in the direction of $V$ and using (\ref{yamabe soliton}), (\ref%
{mert7}), we get%
\begin{eqnarray}
({\mathcal{L}}_{V}S)(X,Y) &=&(-\Delta r+\lambda -r)g(X,Y)+\left( \Delta r+%
\frac{r}{2}(\lambda -r)\right) \eta (X)\eta (Y)  \notag \\
&&-\left( \frac{r}{2}+3\right) \left\{ ({\mathcal{L}}_{V}\eta )(X)\eta (Y)+({%
\mathcal{L}}_{V}\eta )(Y)\eta (X)\right\}  \label{MERT8}
\end{eqnarray}%
for any vector fields $X,Y\in \chi (M)$. In view of (\ref{mert6}) and (\ref%
{MERT8}), we obtain%
\begin{eqnarray}
g(\nabla _{X}D_{r},Y) &=&(-\Delta r+2(\lambda -r))g(X,Y)+(2\Delta
r+r(\lambda -r))\eta (X)\eta (Y)  \notag \\
&&-(r+6)\left\{ ({\mathcal{L}}_{V}\eta )(X)\eta (Y)+({\mathcal{L}}_{V}\eta
)(Y)\eta (X)\right\}  \label{mert9}
\end{eqnarray}%
for any vector fields $X,Y\in \chi (M)$. Setting $X=Y=\xi $ in (\ref{mert9})
and using (\ref{mert4}), (\ref{mer4.1}),we get%
\begin{equation}
0=\Delta r+4(r-\lambda ).  \label{mert10}
\end{equation}%
From the last equation, the proof ends.
\end{proof}

\begin{theorem}
\label{5}If the semi-Riemannian metric of a three-dimensional para-Sasakian
manifold $(M^{3},\varphi ,\xi ,\eta ,g)$ is a Yamabe soliton, then it has
constant scalar curvature, and the vector field $V$ is Killing. Furthermore,
either $(M^{3},\varphi ,\xi ,\eta ,g)$ has constant curvature -1 and reduces
to an Einstein manifold, or $V$ is an infinitesimal automorphism of the
paracontact metric structure on $(M^{3},\varphi ,\xi ,\eta ,g).$
\end{theorem}

\begin{proof}
If we differentiate covariantly (\ref{mer4.1}) along the direction of an
arbitrary vector field $X$ and use (\ref{nablaksi}), we have%
\begin{equation}
\varphi X(r)=g(\nabla _{X}D_{r},\xi ).  \label{mert11}
\end{equation}%
Putting $\xi $ for $Y$ in (\ref{mert9}), using (\ref{mert4}), (\ref{mert10}%
), (\ref{mert11}), we obtain%
\begin{equation}
(\lambda -r)\left( \frac{r}{2}+3\right) \eta (X)-\varphi X(r)=(r+6)({%
\mathcal{L}}_{V}\eta )(X)  \label{mert 11.5}
\end{equation}%
for any vector field $X\in \chi (M)$. Making use of last equation and (\ref%
{mert10}) in (\ref{mert9}), we deduce%
\begin{equation}
\nabla _{X}D_{r}=-2(\lambda -r)(X-\eta (X)\xi )+\varphi X(r)\xi -\eta
(X)\varphi Dr  \label{mert12}
\end{equation}%
for any vector field $X\in \chi (M)$.

Differentiating covariantly along the direction of $Y$, we get%
\begin{eqnarray}
\nabla _{Y}\nabla _{X}D_{r} &=&2Y(r)\left[ X-\eta (X)\xi \right]
\label{mert13} \\
&&-2(\lambda -r)\left[ \nabla _{Y}X+g(\varphi Y,X)\xi -g(\nabla _{Y}X,\xi
)\xi +\eta (X)\varphi Y\right]  \notag \\
&&+Y(\varphi X(r))\xi -(\varphi X(r))\varphi Y  \notag \\
&&-\eta (X)\left[ -g(Y,Dr)\xi +\varphi \nabla _{Y}Dr\right]  \notag \\
&&-\left[ -g(\varphi Y,X)+g(\nabla _{Y}X,\xi )\right] \varphi Dr\text{.} 
\notag
\end{eqnarray}%
Replacing $X$ and $Y$ in (\ref{mert13}), we obtain%
\begin{eqnarray}
\nabla _{X}\nabla _{Y}D_{r} &=&2X(r)\left[ Y-\eta (Y)\xi \right]
\label{MERT14} \\
&&-2(\lambda -r)\left[ \nabla _{X}Y+g(\varphi X,Y)\xi -g(\nabla _{X}Y,\xi
)\xi +\eta (Y)\varphi X\right]  \notag \\
&&+X(\varphi Y(r))\xi -(\varphi Y(r))\varphi X  \notag \\
&&-\eta (Y)\left[ -g(X,Dr)\xi +\varphi \nabla _{X}Dr\right]  \notag \\
&&-\left[ -g(\varphi X,Y)+g(\nabla _{X}Y,\xi )\right] \varphi Dr\text{.} 
\notag
\end{eqnarray}%
From (\ref{mert12}), we have%
\begin{eqnarray}
\nabla _{\left[ Y,X\right] }D_{r} &=&-2(\lambda -r)\left[ \nabla
_{Y}X-\nabla _{X}Y-g(\nabla _{Y}X,\xi )\xi +g(\nabla _{X}Y,\xi )\xi \right]
\label{MERT15} \\
&&+g(\varphi \nabla _{Y}X,Dr)\xi -g(\varphi \nabla _{X}Y,Dr)\xi  \notag \\
&&-\left[ g(\nabla _{Y}X,\xi )-g(\nabla _{X}Y,\xi )\right] \varphi Dr\text{.}
\notag
\end{eqnarray}%
If we put (\ref{mert13}), (\ref{MERT14}) and (\ref{MERT15}) in the
Riemannian curvature tensor $R$ equation $(R(X,Y)D_{r}=\nabla _{X}\nabla
_{Y}D_{r}-\nabla _{Y}\nabla _{X}D_{r}-\nabla _{\left[ X,Y\right] }D_{r})$
and contracting over $Y$ (we assume $(e_{i})$ $(i=1,2,3)$ to be a local
orthonormal frame on $M$), we obtain%
\begin{eqnarray}
S(X,D_{r}) &=&\dsum\limits_{i=1}^{3}\varepsilon _{i}g(R(e_{i},X)D_{r},e_{i})
\notag \\
&=&-\eta (X)\dsum\limits_{i=1}^{3}g(\varphi \nabla _{e_{i}}D_{r},e_{i})
\label{mert16}
\end{eqnarray}%
where $i$ is summer over $1,2,3$. Using (\ref{mert12}) in the last equation,
we have $S(X,D_{r})=0$. Taking account (\ref{MERT5}), we have $(r+2)X_{r}=0$
which shows that $r$ is constant. By virtue of (\ref{mert10}), we get $%
\lambda =r$. The equation (\ref{yamabe soliton}) leads to $\mathcal{L}%
_{V}g=0 $, namely $V$ is Killing.

If $r=-6$, from (\ref{MERT5}), we have $S=-2g(X,Y)$, hence $M$ is an
Einstein manifold. Taking account of (\ref{constant curvature}) and (\ref%
{THREE DIM CURVATURE}), $M$ has constant curvature -1.

If $r\neq -6$, (\ref{mert 11.5}) gives $\mathcal{L}_{V}\eta =0$. By virtue
of this and being $V$ is Killing, we have $\mathcal{L}_{V}\xi =0$. If we
take the Lie-derivative of the well known equation$\ \Phi (X,Y)=g(X,\varphi
Y)$ in the direction of $V$, we\ get $\mathcal{L}_{V}\Phi =0$. From (\ref%
{infi}), we can conclude that $V$ is an infinitesimal automorphism of the
paracontact metric structure on $(M^{3},\varphi ,\xi ,\eta ,g)$. So, the
proof ends.
\end{proof}

\begin{proposition}
\label{5.5}Let $(M^{3},\varphi ,\xi ,\eta ,g)$ be a three-dimensional
para-Sasakian manifold. If $M^{3}$ admits a Yamabe soliton and $V$ is
pointwise collinear vector field with the structure vector field $\xi $,
then $V$ is a constant multiple of $\xi $.
\end{proposition}

\begin{proof}
Let $V$ be a pointwise collinear vector field with the structure vector
field $\xi $, that is $V=b\xi $, where $b$ is a smooth function on $M^{3}$.
From (\ref{yamabe soliton}), we obtain%
\begin{equation}
g(\nabla _{X}V,Y)+g(\nabla _{Y}V,X)=0  \label{mert 17}
\end{equation}%
for any vector fields $X,Y\in \chi (M)$.

Taking $V=b\xi $ in (\ref{mert 17}) and using (\ref{nablaksi}), we have%
\begin{equation}
X(b)\eta (Y)+Y(b)\eta (X)=0  \label{MERT 18}
\end{equation}%
for any vector fields $X,Y\in \chi (M)$. Putting $Y=\xi $ in (\ref{MERT 18}%
), we get%
\begin{equation}
X(b)+\xi (b)\eta (X)=0\text{.}  \label{mert19}
\end{equation}%
Replacing $X$ by $\xi $ in the last eqaution, we obtain%
\begin{equation}
\xi (b)=0.  \label{mert20}
\end{equation}%
If we use (\ref{mert20}) in (\ref{mert19}), we have%
\begin{equation*}
X(b)=0
\end{equation*}%
which yields $db=0$, that is, $b=$constant. This completes the proof.
\end{proof}

\begin{theorem}
\label{6}If the semi-Riemannian metric of a three dimensional para-Sasakian
manifold with $r=-6$ is a Yamabe soliton, then it is also a Ricci soliton.
\end{theorem}

\begin{proof}
Let $(M^{3},\varphi ,\xi ,\eta ,g)$ be a three-dimensional para-Sasakian
manifold. From Theorem \ref{5} one can say that if $g$ is a Yamabe soliton
with $r=-6$ then the Ricci operator of $(M^{3},\varphi ,\xi ,\eta ,g)$ is $%
Q=-2id$. Therefore (\ref{Ricci Soliton}) returns to%
\begin{equation*}
(\mathcal{L}_{V}g)(X,Y)+2S(X,Y)+2\mu g(X,Y)=0.
\end{equation*}%
Hence, $g$ is an expanding Ricci soliton with $\mu =2$.
\end{proof}

\section{Yamabe Solitons on three dimensional Paracosymplectic manifolds}

An almost paracontact metric manifold $M^{2n+1}$, with a\ structure $%
(\varphi ,\xi ,\eta ,g)$ is said to be an \textit{almost }$\alpha $\textit{%
-paracosymplectic} manifold, if 
\begin{equation}
d\eta =0,\quad d\Phi =2\alpha \eta \wedge \Phi ,  \label{genel}
\end{equation}%
where $\alpha $ may be a constant or function on $M.$

For a particular choices of the function $\alpha =0$, we have \textit{almost
paracosymplectic} manifolds. If additionaly normality conditon is fulfilled,
then manifolds are called \textit{paracosymplectic}. We may refer to \cite%
{DACKO} and \cite{piotr} and references therein for more information about
paracosymplectic manifolds. We also recall that any paracosymplectic
manifold satisfies 
\begin{eqnarray}
R(X,Y)\xi &=&0,  \label{c1} \\
(\nabla _{X}\varphi )Y &=&0,  \label{c2} \\
\nabla _{X}\xi &=&0,  \label{c3} \\
S(X,\xi ) &=&0,  \label{c5}
\end{eqnarray}%
where $Q$ is the Ricci operator, $R$ is the Riemannian curvature tensor and $%
S$ is Ricci tensor defined by $S(X,Y)=g(QX,Y).$

Lemma \ref{2} also valid for paracosymplectic manifolds. For a
three-dimensional paracosymplectic manifold, using (\ref{c1}), (\ref{c5}) in
(\ref{THREE DIM CURVATURE}), we have%
\begin{eqnarray}
QX &=&\frac{r}{2}(X-\eta (X)\xi ),  \label{c6} \\
S(X,Y) &=&\frac{r}{2}(g(X,Y)-\eta (X)\eta (Y))  \label{C7}
\end{eqnarray}%
for any vector field $X\in \chi (M)$. If we use (\ref{c6}) and use the same
procedure in the proof of Lemma \ref{2.5}, we obtain $\xi (r)=0.$

\begin{lemma}
\label{1.}For a three-dimensional paracosymplectic manifold which admits a
Yamabe soliton has always harmonic scalar curvature, that is $\Delta r=0$.
\end{lemma}

\begin{proof}
If we take the Lie-derivative of (\ref{C7}) in the direction of $V$ and
using (\ref{yamabe soliton}), (\ref{mert7}), we get%
\begin{eqnarray}
({\mathcal{L}}_{V}S)(X,Y) &=&(-\Delta r)g(X,Y)+\left( \Delta r+\frac{r}{2}%
(\lambda -r)\right) \eta (X)\eta (Y)  \notag \\
&&-\frac{r}{2}\left\{ ({\mathcal{L}}_{V}\eta )(X)\eta (Y)+({\mathcal{L}}%
_{V}\eta )(Y)\eta (X)\right\}  \label{C8}
\end{eqnarray}%
for any vector fields $X,Y\in \chi (M)$. In view of (\ref{mert6}) and (\ref%
{C8}), we obtain%
\begin{eqnarray}
g(\nabla _{X}D_{r},Y) &=&(-\Delta r)g(X,Y)+(2\Delta r+r(\lambda -r))\eta
(X)\eta (Y)  \notag \\
&&-r\left\{ ({\mathcal{L}}_{V}\eta )(X)\eta (Y)+({\mathcal{L}}_{V}\eta
)(Y)\eta (X)\right\}  \label{C9}
\end{eqnarray}%
for any vector fields $X,Y\in \chi (M)$. Setting $X=Y=\xi $ in (\ref{C9})
and using (\ref{mert4}), $\xi (r)=0$, we get%
\begin{equation}
\Delta r=0.  \label{C10}
\end{equation}
\end{proof}

Unlike the case of contact, we can not deduce that $r$ is constant. In other
words, on compact semi-Riemannian manifold, there may be non-constant $r$
such that $\Delta r=0$. So we prove following:

\begin{theorem}
\label{2.}If the semi-Riemannian metric of a three-dimensional
paracosymplectic manifold $(M^{3},\varphi ,\xi ,\eta ,g)$ is a Yamabe
soliton, then it has constant scalar curvature. So, if $r\neq 0$, then
manifold is $\eta $-Einstein, if $r=0$, then manifold is Ricci flat.
\end{theorem}

\begin{proof}
If we differentiate covariantly $\xi (r)=0$ along the direction of an
arbitrary vector field $X$ and use (\ref{c3}), we have%
\begin{equation}
0=g(\nabla _{X}D_{r},\xi ).  \label{c11}
\end{equation}%
Putting $\xi $ for $Y$ in (\ref{C9}), using (\ref{mert4}), (\ref{C10}), (\ref%
{c11}), we obtain%
\begin{equation}
\left( (\lambda -r)\frac{r}{2}\right) \eta (X)=r({\mathcal{L}}_{V}\eta )(X)
\label{c12}
\end{equation}%
for any vector field $X\in \chi (M)$. Making use of last equation and (\ref%
{C10}) in (\ref{C9}), we deduce%
\begin{equation}
\nabla _{X}D_{r}=0  \label{c13}
\end{equation}%
for any vector field $X\in \chi (M)$. If we use (\ref{c13}) in the
Riemannian curvature tensor $R$ equation $(R(X,Y)D_{r}=\nabla _{X}\nabla
_{Y}D_{r}-\nabla _{Y}\nabla _{X}D_{r}-\nabla _{\left[ X,Y\right] }D_{r})$,
we obtain $R(X,Y)D_{r}=0$. So, it is clear that $S(X,D_{r})=0$. Taking
account (\ref{C7}), we have $\frac{r}{2}X(r)=0$ which shows that $r$ is
constant. By (\ref{C7}), one can easily deduce that if $r\neq 0$, then
manifold is $\eta $-Einstein, if $r=0$, then manifold is Ricci flat.
\end{proof}

\section{Yamabe Solitons on three dimensional Para-Kenmotsu manifolds}

For a particular choices of the function $\alpha =1$ in (\ref{genel}) we
have \textit{almost para-Kenmotsu} manifolds. If additionaly normality
conditon is fulfilled, then manifolds are called \textit{para-Kenmotsu.}. We
may refer to \cite{piotr} and references therein for more information about
para-Kenmotsu manifolds. We also recall that any para-Kenmotsu manifold
satisfies 
\begin{eqnarray}
R(X,Y)\xi &=&\eta (X)Y-\eta (Y)X,  \label{K1} \\
(\nabla _{X}\varphi )Y &=&g(\varphi X,Y)\xi -\eta (Y)\varphi X,  \label{K2}
\\
\nabla _{X}\xi &=&X-\eta (X)\xi ,  \label{K3} \\
S(X,\xi ) &=&-(n-1)\eta (X),  \label{K4}
\end{eqnarray}%
where $Q$ is the Ricci operator, $R$ is the Riemannian curvature tensor and $%
S$ is Ricci tensor defined by $S(X,Y)=g(QX,Y).$

Lemma \ref{2} also valid for para-Kenmotsu manifolds.

\begin{lemma}
\label{LK1}For any three-dimensional para-Kenmotsu manifold $(M^{3},\varphi
,\xi ,\eta ,g)$, we have%
\begin{equation}
\xi (r)=-2(r+6).  \label{K5}
\end{equation}
\end{lemma}

\begin{proof}
If we replace $Y=Z$ by $\xi $ in (\ref{THREE DIM CURVATURE}) and use (\ref%
{K1}), (\ref{K4}) we get%
\begin{equation}
QX=\left( \frac{r}{2}+1\right) X-\left( \frac{r}{2}+3\right) \eta (X)\xi
\label{K6}
\end{equation}%
for any vector field $X\in \chi (M)$. If we use (\ref{K6}), (\ref{K3}) and (%
\ref{K4}) in the following well known formula for semi-Riemannian manifolds 
\begin{equation*}
trace\left\{ Y\rightarrow (\nabla _{Y}Q)X\right\} =\frac{1}{2}\nabla _{X}r
\end{equation*}%
we obtain the requested equation.
\end{proof}

Since the proofs of the following lemma and theorem are quite similar to
Lemma 3.3 and Theorem 1.1 of \cite{wang}, so we don't give the proofs of
them.

\begin{lemma}
\label{LK2}If the semi-Riemannian metric on a three-dimensional
para-Kenmotsu manifold $(M^{3},\varphi ,\xi ,\eta ,g)$ is a Yamabe soliton,
then the Yamabe soliton is expanding with $\lambda =-6$ and the scalar
curvature of $M^{3}$ is harmonic, that is $\Delta r=0$.
\end{lemma}

\begin{theorem}
\label{TK1}If the semi-Riemannian metric on a three-dimensional
para-Kenmotsu manifold $(M^{3},\varphi ,\xi ,\eta ,g)$ is a Yamabe soliton,
then the manifold is of constant sectional curvature $-1$, reduces to an
Einstein manifold. Furthermore, Yamabe soliton is expanding with $\lambda
=-6 $ and the vector field $V$ is Killing.
\end{theorem}

The proof of following theorem is similar to Theorem \ref{6}.

\begin{theorem}
\label{TK2}If the semi-Riemannian metric of a three dimensional
para-Kenmotsu manifold is a Yamabe soliton, then it is also a Ricci soliton.
\end{theorem}

\begin{proposition}
\label{PK1}Let $(M^{3},\varphi ,\xi ,\eta ,g)$ be a three-dimensional
para-Kenmotsu manifold. If $M^{3}$ admits a Yamabe soliton for a vector
field $V$ and a constant $\lambda $, then $V$ can not be pointwise collinear
with $\xi $.
\end{proposition}

\begin{proof}
Let $V$ be a pointwise collinear vector field with the structure vector
field $\xi $, that is $V=b\xi $, where $b$ is a smooth function on $M^{3}$.
From (\ref{yamabe soliton}), we obtain%
\begin{equation}
g(\nabla _{X}V,Y)+g(\nabla _{Y}V,X)=0  \label{K7}
\end{equation}%
for any vector fields $X,Y\in \chi (M)$.

Taking $V=b\xi $ in (\ref{K7}) and using (\ref{K3}), we have%
\begin{equation}
X(b)\eta (Y)+Y(b)\eta (X)+2bg(X,Y)-2b\eta (X)\eta (Y)=0  \label{K8}
\end{equation}%
for any vector fields $X$ and $Y$. We have $f=0$, for $X$ and $Y$ belongs to
the contact distribution, in the last equation.
\end{proof}

\section{EXAMPLES}

Now, we will give examples which support Theorem \ref{5} and Theorem \ref%
{TK1} .

\begin{example}
\label{ex1} We consider the $3$-dimensional manifold%
\begin{equation*}
M^{3}=\{(x,y,z)\in 
\mathbb{R}
^{3},z\neq 0\}
\end{equation*}%
and the vector fields%
\begin{equation*}
\varphi e_{2}=e_{1}=2y\text{ }\frac{\partial }{\partial x}+z\frac{\partial }{%
\partial z},\text{ \ \ }\varphi e_{1}=e_{2}=\frac{\partial }{\partial y},%
\text{ \ \ }\xi =e_{3}=\frac{\partial }{\partial x}.
\end{equation*}%
The $1$-form $\eta =-\frac{2y}{z}dz$ defines an almost paracontact structure
on $M$ with characteristic vector field $\xi =\frac{\partial }{\partial x}$.
Let $g$, $\varphi $ be the semi-Riemannian metric and the $(1,1)$-tensor
field respectively given by 
\begin{eqnarray*}
g &=&\left( 
\begin{array}{ccc}
1 & 0 & -\frac{y}{z} \\ 
0 & -1 & 0 \\ 
-\frac{y}{z} & 0 & \frac{1+4y^{2}}{z^{2}}%
\end{array}%
\right) ,\text{ } \\
\varphi \text{\ } &=&\left( 
\begin{array}{ccc}
0 & 2y & 0 \\ 
0 & 0 & \frac{1}{z} \\ 
0 & z & 0%
\end{array}%
\right) ,
\end{eqnarray*}%
with respect to the basis $\frac{\partial }{\partial x},\frac{\partial }{%
\partial y},\frac{\partial }{\partial z}$.
\end{example}

Using (\ref{nablaksi}) we have%
\begin{equation*}
\begin{array}{ccc}
\nabla _{e_{1}}e_{1}=0,~ & \nabla _{e_{2}}e_{1}=\xi , & ~~~\nabla _{\xi
}e_{1}=-e_{2}, \\ 
\nabla _{e_{1}}e_{2}=-\xi , & \nabla _{e_{2}}e_{2}=0,~ & \nabla _{\xi
}e_{2}=-e_{1}, \\ 
\nabla _{e_{1}}\xi =-e_{2}, & ~~\nabla _{e_{2}}\xi =-e_{1}, & \nabla _{\xi
}\xi =0.%
\end{array}%
\end{equation*}%
Hence the manifold is a 3-dimensional para-Sasakian manifold. One can easily
compute, 
\begin{equation}
\begin{array}{ccc}
R(e_{1},e_{2})\xi =0, & R(e_{2},\xi )\xi =-e_{2}, & R(e_{1},\xi )\xi =-e_{1},
\\ 
R(e_{1},e_{2})e_{2}=-3e_{1}, & R(e_{2},\xi )e_{2}=-\xi , & R(e_{1},\xi
)e_{2}=0, \\ 
R(e_{1},e_{2})e_{1}=-3e_{2}, & R(e_{2},\xi )e_{1}=0, & R(e_{1},\xi
)e_{1}=\xi .%
\end{array}
\label{r1}
\end{equation}%
Using (\ref{r1}), we have constant scalar curvature as follows,%
\begin{equation*}
r=S(e_{1},e_{1})-S(e_{2},e_{2})+S(\xi ,\xi )=2.
\end{equation*}%
Because of scalar curvature $r\neq -6$, we can conclude that $V$ is an
infinitesimal automorphism of the paracontact metric structure on $M^{3}$.

\begin{example}
\label{ex2} We consider the $3$-dimensional manifold%
\begin{equation*}
M=\{(x,y,z)\in 
\mathbb{R}
^{3},z\neq 0\}
\end{equation*}%
and the vector fields%
\begin{equation*}
X=\text{ }\frac{\partial }{\partial x},\text{ \ \ }\varphi X=\frac{\partial 
}{\partial y},\text{ \ \ }\xi =(x+2y)\frac{\partial }{\partial x}+(2x+y)%
\frac{\partial }{\partial y}+\frac{\partial }{\partial z}.
\end{equation*}%
The $1$-form $\eta =dz$ defines an almost paracontact structure on $M$ with
characteristic vector field $\xi =(x+2y)\frac{\partial }{\partial x}+(2x+y)%
\frac{\partial }{\partial y}+\frac{\partial }{\partial z}$. Let $g$, $%
\varphi $ be the semi-Riemannian metric and the $(1,1)$-tensor field given
by 
\begin{eqnarray*}
g &=&\left( 
\begin{array}{ccc}
1 & 0 & -\frac{1}{2}(x+2y) \\ 
0 & -1 & \frac{1}{2}(2x+y) \\ 
-\frac{1}{2}(x+2y) & \frac{1}{2}(2x+y) & 1-(2x+y)^{2}+(x+2y)^{2}%
\end{array}%
\right) ,\text{ } \\
\varphi \text{\ } &=&\left( 
\begin{array}{ccc}
0 & 1 & -(2x+y) \\ 
1 & 0 & -(x+2y) \\ 
0 & 0 & 0%
\end{array}%
\right) ,
\end{eqnarray*}%
with respect to the basis $\frac{\partial }{\partial x},\frac{\partial }{%
\partial y},\frac{\partial }{\partial z}$.
\end{example}

Using (\ref{K3}) we have%
\begin{equation*}
\begin{array}{ccc}
\nabla _{X}X=-\xi ,~ & \nabla _{\varphi X}X=0, & ~~~\nabla _{\xi
}X=-2\varphi X, \\ 
\nabla _{X}\varphi X=0, & \nabla _{\varphi X}\varphi X=\xi ,~ & \nabla _{\xi
}\varphi X=-2X, \\ 
\nabla _{X}\xi =X, & ~~\nabla _{\varphi X}\xi =\varphi X, & \nabla _{\xi
}\xi =0.%
\end{array}%
\end{equation*}%
Hence the manifold is a para-Kenmotsu manifold. One can easily compute, 
\begin{equation}
\begin{array}{ccc}
R(X,\varphi X)\xi =0, & R(\varphi X,\xi )\xi =-\varphi X, & R(X,\xi )\xi =-X,
\\ 
R(X,\varphi X)\tilde{\varphi}X=X, & R(\varphi X,\xi )\tilde{\varphi}X=-\xi ,
& R(X,\xi )\varphi X=0, \\ 
R(X,\varphi X)X=\varphi X, & R(\varphi X,\xi )X=0, & R(X,\xi )X=\xi .%
\end{array}
\label{r2}
\end{equation}%
Using (\ref{r2}), we have constant scalar curvature as follows,%
\begin{equation*}
\tau =S(X,X)-S(\varphi X,\varphi X)+S(\xi ,\xi )=-6.
\end{equation*}%
Because of scalar curvature $r=-6$, from Theorem \ref{TK1}, we can conclude
that $M$ is an Einstein manifold.

\end{document}